\documentclass[11pt,leqno]{amsart}
\usepackage{amscd}
\usepackage{color}
\usepackage{amssymb}
\usepackage{amsfonts}
\usepackage{latexsym}
\usepackage{verbatim}
\usepackage{dsfont}

\theoremstyle{plain}
\newtheorem{theorem}{Theorem}[section]

\newtheorem{lemma}[theorem]{Lemma}
\newtheorem{prop}[theorem]{Proposition}
\newtheorem{cor}[theorem]{Corollary}
\newtheorem{rem}[theorem]{Remark}

\title{On the convergence of the Sasaki $J$--flow}
\author[M. Zedda]{Michela Zedda}
\date{\today}
\subjclass[2010]{53C25; 53C44}
\keywords{Sasakian manifolds, geometric flows, critical metrics}
\address{Dipartimento di Matematica e Informatica\\ Universit\`a di Parma} 
 \email{michela.zedda@gmail.com}
\thanks{This work was supported by the project FIRB ``Geometria differenziale e teoria geometrica delle funzioni'' and by G.N.S.A.G.A. of I.N.d.A.M} 

\begin{document}
\maketitle
\begin{abstract}
This paper investigates the $C^\infty$-convergence of the Sasaki $J$-flow. The result is applied to prove a lower bound for the $K$-energy map in the Sasakian context.
\end{abstract}
\section{Introduction and statement of the main results}
%In this paper we study the convergence of the {\em Sasaki $J$-flow}. 
The Sasaki $J$-flow has been introduced in \cite{VZ} as a natural counterpart of the K\"ahler $J$-flow in the Sasakian setting.
In \cite[pages 10-11]{donaldsonMM} Donaldson introduced the K\"ahler $J$-flow on a $n$-dimensional K\"ahler manifold $M$ pointing out the importance of its critical points from the point of view of moment maps. Given two K\"ahler forms $\omega$ and $\chi$ on $M$, such critical points satisfy the equation $n\chi\wedge\omega^{n-1}=c\,\omega^n$, where $c$ is a constant depending on $[\omega]$ and $[\chi]$. Donaldson observed that a necessary condition for their existence is that $[c\,\omega-\chi]$ be a K\"ahler class, and he asked if it was also sufficient. In \cite{chenM} X.X. Chen proved that for complex surfaces it is. Although, it is not in higher dimension, indeed in the recent paper \cite{LSS}, M. Lejmi and G. Sz\'ekelyhidi found a counterexpample on the blow up of $\mathds{P}^3$ at one point. The existence of critical points has been studied by B. Weinkove, V. Tosatti and J. Song in \cite{songweinkoveJ, tosatti} in terms of the positivity of some $(n-1,n-1)$ form and by X.X. Chen in \cite{chen} in terms of the sign of the holomorphic sectional curvature of $\chi$. In particular, in \cite{chen} Chen proved the long time existence of the flow and, when the bisectional curvature of $\chi$ is nonpositive, its convergence to a critical metric.

This work was inspired by \cite{wein1}, where B. Weinkove deals 
with the natural question on what the behaviour of the flow is on K\"ahler surfaces, where the existence of a critical metric is always guaranteed (once the necessary condition above is satisfied). He proved the convergence of the $J$--flow on K\"ahler surfaces under the only assumption $c\,\omega-\chi$ to be positive.
In the later work \cite{wein2}, Weinkove generalize his result to higher dimension, proving the convergence of the K\"ahler $J$-flow under the assumption of the positivity of the form $c\,\omega-(n-1)\chi$.

In the Sasakian case the situation is quite similar and a necessary condition for the existence of a critical point is that there exists a basic map $h$ such that $\frac c2\,(d\eta+dd^ch)-\chi$ is a transverse K\"ahler form (see \cite{VZ} or next section below for definitions and details). In \cite[Prop. 3.3]{VZ}, it is proven that that condition is also sufficient on $5$-dimensional Sasakian manifolds.

The first result of this paper is the following theorem, which translates Weinkove's result in \cite{wein2} to the Sasakian context (we remaind to Section \ref{def} for definitions and details):
\begin{theorem}\label{main}
Let  $(M,\xi,\Phi,\eta,g)$ be a compact $2n+1$-dimensional Sasaki manifold. Assume that $\frac c2\,d\eta-(n-1)\chi$ is a transverse K\"ahler form. Then, the Sasaki $J$--flow converges $C^\infty$ to a smooth critical metric. 
\end{theorem}
The proof is based on the long time existence of the flow and the uniform lower bound on the second derivatives of a solution to the flow established in \cite{VZ}, and on the estimates developed in sections \ref{secord} and \ref{c0} of the present paper, which % similarly to \cite{wein1, wein2}, 
are obtained applying the maximum principle and a Moser iteration argument (see the proof of Prop \ref{normac0}).

% the existence of critical metrics for $5$--dimensional Sasakian manifolds (see \cite[Prop. 3.3]{VZ}) and on the existence of an uniform upper bound for the second derivatives of a solution of the flow (Proposition \ref{secondorder} below). 

It is worth pointing out that as immediate corollary we get the $C^\infty$ convergence of the flow to a critical metric on compact $5$-dimensional Sasaki manifolds under the assumption\linebreak  {$\frac c2d\eta-\chi>0$}. 

%\begin{cor}
%Let  $(M,\xi,\Phi,\eta,g)$ be a compact $5$-dimensional Sasaki manifold. Assume that there exists a map $h\in \mathcal{H}$ such that  $\frac c2\,(d\eta+dd^ch)-\chi$ is a transverse K\"ahler form. Then, the Sasaki $J$--flow converges $C^\infty$ to a smooth critical metric. 
%\end{cor}
As application, we highlight the relation between the Sasaki $J$-flow and the Mabuchi $K$-energy, introduced in the Sasakian context by \cite{F} (see also \cite{Z}), proving a lower bound for the $K$-energy map under the existence of a critical metric (see Theorem \ref{second} at the end of the paper). \\
%\begin{cor}
%If $\rho
%\end{cor}

%The importance of critical points for the Sasaki $J$-flow has been highlight in \cite{}, where

The paper contains three more sections. In the first one we summarize some basic facts about Sasakian geometry, recall the definition of Sasaki $J$-flow and set the notations. In the second one we develope second order estimates on a solution $f$ to the $J$-flow which depends on $f$ itself. In the third section we study the $C^0$-estimates and prove Theorem \ref{main}. Finally in the fourth and last section we recall the defintion of Mabuchi $K$-energy in the Sasakian context and prove our second result Theorem \ref{second}.\\

The author is very grateful to Luigi Vezzoni for suggesting her to study critical points of the $J$--flow in the Sasakian context, for his advices and for all the interesting discussions along the preparation of this work.% and for pointing out some mistakes in the earlier versions of this paper.

\section{Sasakian manifolds and the Sasaki $J$-flow}\label{def}
Here we briefly recall what we need about Sasakian manifolds, the reader is referred to \cite{bgbook, sparks} for a more detailed exposition.

A Riemannian manifold $(M,g)$ is Sasakian if and only if the Riemannian cone $(M\times \mathds{R}^+,\bar g=r^2g+dr^2)$ is K\"ahler. The integrable complex structure $J$ and the K\"ahler form $\bar \omega$ on $(M\times \mathds{R}^+,\bar g)$ induce in a natural way on $(M,g)$:
\begin{enumerate}
\item a killing vector field $\xi$ and its dual $1$-form $\eta$, $\eta(\xi)=1$, $\iota_\xi d\eta=0$, which is a contact form, i.e. $\eta\wedge (d\eta)^n\neq 0$. The tangent bundle $TM$ splits into $TM=D\oplus L_\xi$, where $D=\ker \eta$ and $L_\xi$ is the line tangent to $\xi$.
\item an endomorphism $\Phi$ defined by $\Phi|_D=J|_D$, $\Phi_{L_\xi}=0$, which satisfies $\Phi^2=-{\rm Id}+\eta\otimes\xi$.
\end{enumerate}
The triple $(\eta,\xi,\Phi)$ realises a contact structure on $M$, while $(D,\Phi|_D)$ realises a CR structure.

According to $TM=D\oplus L_\xi$, the metric $g$ splits into:
$$
g(X,Y)=g^T(X,Y)+\eta(X)\eta(Y), \quad X, Y\in TM,
$$
where $g^T(X,Y)=\frac12 d\eta(X,\Phi Y)$, which is zero along the direction of $\xi$ and it is K\"ahler with respect to $D$, is called the {\em transverse K\"ahler metric} of $M$.

A $p$-form $\alpha$ on $M$ is basic if it satisfies:
$$
\iota_\xi\alpha=0,\quad \iota_\xi d\alpha=0.
$$
In particular, a function is basic if and only if its derivatives in the direction of $\xi$ vanishes. We denote the space of smooth basic functions on $M$ by $C_B^\infty (M,\mathds{R})$.

Given a Sasakian manifold $(M,g,\xi,\Phi,\eta)$, consider
$$
\mathcal{H}=\{ f\in C^\infty_B(M,\mathds{R})|\, \eta_f=\eta+d^cf\, \textrm{is a contact form}\},
$$
where $(d^cf)(X)=-\frac12df(\Phi(X))$ for any vector field $X$ on $M$. Observe that any $f\in \mathcal H$ induces a Sasakian structure $(\xi, \Phi_f, \eta_f)$ on $M$ with the same Reeb vector field $\xi$. The geometry of $\mathcal H$ has been studied by P. Guan and X. Zhang in \cite{guanzhang, guanzhangA} from  the point of view of geodesics, by W. He in \cite{he} from the point of view of curvature and by S. Calamai, D. Petrecca and K. Zheng in \cite{cpz} in relation to the Ebin metric.
  
In order to define the $J$-flow, we need to fix a transverse K\"ahler form $\chi$ on $M$, i.e. $\chi$ is a basic $(1,1)$-form which is positive and closed.
Let $f(t)$ be a smooth path on $\mathcal H$. Define the functional $J_\chi\!:\mathcal H\rightarrow \mathds{R}$ by:
 $$
 (\partial_t J_\chi)(f)=\frac{1}{2^{n-1}(n-1)!}\int_M\dot f\,\chi\wedge\eta\wedge (d\eta_f)^{n-1}=\frac{1}{2^{n}n!}\int_M\dot f\,\sigma_f\,\eta\wedge (d\eta_f)^{n},\quad J_\chi(0)=0,
$$
where $\sigma_f$ is the trace of $\chi$ with respect to $d\eta_f$. Alternatively, the $J_\chi$ functional can be defined by (cfr. \cite[Def. 2.2]{VZ}):
\begin{equation}\label{altj}
J_{\chi}(h)=\frac{1}{2^{n-1}(n-1)!} A_{\chi}(0,h)\,\nonumber
\end{equation}
where
$$
A_{\chi}(f):=\int_0^1\int_M \dot f\,\chi\wedge \eta\wedge (d\eta_f)^{n-1}\,dt.
$$
Further, let 
$\mathcal H_0=\{h\in \mathcal H|\ I(h)=0\}$, where $I\!:\mathcal H\rightarrow \mathds{R}$ is defined by:
$$
(\partial_tI)(f)=\int_M\dot f\eta\wedge (d\eta_f)^n,\quad I(0)=0.
$$
Notice that $I$ can be explicitely written by (see \cite[Eq. 14]{guanzhang}):
\begin{equation}\label{h0}
I(f)=\sum_{p=0}^n\frac{n!}{(p+1)!(n-p)!}\int_Mf\,\eta\wedge(d\eta_0)^{n-p}\wedge (i\partial_B\bar\partial_Bf)^p.
\end{equation}
Observe that $h\in \mathcal H_0$ is a critical point of $J_{\chi}$ restricted to $\mathcal{H}_0$ if and only if 
$$
\int_M k\,\eta\wedge \chi\wedge (d\eta_h)^{n-1}=0,
$$
for  every $k\in T_h\mathcal{H}_0$, i.e. if and only if $2n \,\eta\wedge \chi\wedge (d\eta_h)^{n-1}=c\, \eta\wedge (d\eta_h)^{n}$, where 
$$
c=\frac{2n\int_M \chi \wedge \eta\wedge (d\eta)^{n-1}}{\int_M \eta\wedge (d\eta)^n}\,. 
$$
In particular, $h\in\mathcal{H}_0$ is a critical point  of $J_{\chi}$ iff $\sigma_h=c$. The Sasaki $J$-flow is the gradient flow of $J_{\chi}\colon \mathcal{H}_0\to \mathds{R}$ and its evolution equation is given by:
\begin{equation}\label{jflow}
\dot f=c-\sigma_f\,,\quad f(0)=0.
\end{equation}

In the joint work with L. Vezzoni \cite{VZ}, we study the long time existence of the $J$-flow \eqref{jflow} and prove its convergence to a critical metric under an additional hypothesis on the sign of the transverse holomorphic sectional curvature of $\chi$. In the recent paper \cite{BHV}, the Sasaki $J$-flow is included as particular case in a more general result that prove the short time existence of second order geometric flows on foliated manifolds.

\medskip
We conclude this section recalling the definition of {\em special foliated coordinates} (see \cite[Subsec. 2.1]{VZ}) and setting some notations.

Let $(M,\xi,\Phi_f,\eta_f,g_f)$ be a Sasakian manifold as above and let $\chi$ be a second transverse K\"ahler form on $M$. Around each point $(x_0,t_0)\in M\times [0,+\infty)$ we can find {\em special foliated coordinates}
$\{z^1,\dots,z^{n},z\}$ for $\chi$, taking values in $\mathds{C}^n\times \mathds{R}$, such that 
\begin{equation}\label{coord}
\xi=\partial_z\,,\quad 
\Phi(d{z^j})=i\,d{z^j},\quad \Phi(d{\bar z^j})=-i\,d{\bar z^j}\,,
\end{equation}
and if we denote $\chi=\chi_{i\bar j}\,dz^i\wedge d\bar z^j$,  then 
$$
\chi_{i\bar j}=\delta_{ij}\,,\quad \partial_{z^r}\chi_{i\bar j}=0\,,\mbox{ at }(x_0,t_0)\,,
$$
and 
$$
(g_f)_{j\bar k}=\lambda_j\delta_{jk}\,,\mbox{ at }(x_0,t_0)\,.
$$
Here we denote: 
 $$
g_f=(g_f)_{i\bar j}dz^id\bar z^{j}+\eta_f^2\,,\quad d\eta_f=2i(g_f)_{i\bar j}dz^i\wedge d\bar z^{j},
$$
and in particular the transverse K\"ahler metric $g^T_f$ reads locally $g_f^T=(g_f)_{i\bar j}dz^id\bar z^{ j}$. 

%Moreover, we can further require that the transverse metric $g_f^T$.
%A function $h$ is basic if and only if does not depend on the variable $z$ and we usually denote by 
%$h_{,i_1\dots i_r\bar j_1\dots \bar j_l}$ the space derivatives of $h$ along $\partial_{z^{i_1}},\dots ,\partial_{z^{i_r}},
%\partial_{\bar z^{ j_1}},\dots,\partial_{\bar z^{ j_l}}$. We denote by $A_{i_1\dots i_r\bar j_1\dots \bar j_l}$ (without \lq\lq ,\rq\rq) the components of the basic tensor $A$. Furthermore, when a function $f$ depends also on a time variable $t$, we use notation 
%$\dot f$ to denote its time derivative.  In the case when $f$ depends on two time variables $(t,s)$, we write $\partial_t f$ and $\partial_s f$, to distinguish the two derivatives. 

Observe that $(g_f)_{i\bar j}$ are basic functions and by \eqref{coord}, in these coordinates a function is basic iff it does not depend on $z$. Let us also use upper indexes to denote the entries of a matrix' inverse. Further, for any $h\in C^\infty_B(M)$ we denote $f_{,j}=\partial_j f=\frac{\partial}{\partial z_j}f$, $f_{,\bar j}=\partial_{\bar j} f=\frac{\partial}{\partial \bar z_j}f$, $f_{,j\bar k}=\partial_{j}\partial_{\bar k} f=\frac{\partial^2}{\partial z_j\partial\bar z_k}f$.
In the sequel we will also denote by $\gamma_f$ the trace of $g_f$ with respect to $\chi$, i.e. locally $\gamma_f=\chi^{\bar j k}(g_f)_{j\bar k}$, and by $R(\chi)_{j\bar k l\bar m}$ the curvature tensor associated to $\chi$, which in special coordinates for $\chi$ at $(x_0,t_0)$ reads $R(\chi)_{j\bar k l\bar m}=-\chi_{j\bar k, l\bar m}$. Consequently we will have:
$$
R(\chi)^{a\bar b}_{\ \ l\bar m}=-\chi^{\bar a j}\chi^{\bar k b}\chi_{j\bar k, l\bar m},\quad {\rm Ric}(\chi)_{l\bar m}=-\chi^{\bar k j}\chi_{j\bar k, l\bar m}.
$$
Further, we denote by $(\cdot,\cdot)_\chi$ the product on basic forms $\alpha$, $\beta\in \Omega_B^{(p,q)}(M,\mathds{C})$:
$$
(\alpha,\beta)_\chi=\frac{1}{2^nn!}\int_M \langle \alpha,\beta\rangle_\chi\,\eta\wedge \chi^n\,. 
$$
where
$$
\langle \alpha,\beta\rangle_\chi=\alpha_{i_1\dots i_p\bar j_1\dots\bar j_q}\cdot \bar \beta_{r_1\dots r_p\bar s_1\dots\bar s_q} 
\chi^{\bar r_1i_1}\cdots \chi^{\bar r_pi_p}\cdot \chi^{\bar j_1 s_1}\cdots \chi^{\bar j_q s_q}.
$$

Finally, let $\tilde\Delta_f$ be the operator depending on a smooth curve $f$ in $\mathcal H$ and acting on basic smooth functions, defined by:
$$
\tilde \Delta_f(h)=g_f^{\bar k p}g_f^{\bar q j}\chi_{j\bar k} h_{,a\bar b}\,. 
$$
\begin{rem}\label{delta}\rm
Observe that the operator $\tilde \Delta_f$ satisfies the following property. If $h(x,t)$ is a smooth path in $\mathcal H$ and $(x_0,t_0)$ is a global maximum for $h$ in $M\times [0,t]$, then by the maximum principle at $(x_0,t_0)$ one has:
$$
(\partial_t-\tilde \Delta_f)h\geq 0.
$$
More precisely, at $(x_0,t_0)$ one has $\partial_th\geq 0$ and $dd^ch\leq 0$. Then, by the definition of $\tilde \Delta_f$, one get $\tilde \Delta_f(h)\leq 0$ and thus $\partial_t-\tilde \Delta_f\geq 0$.
\end{rem}
\begin{rem}\label{lambda}\rm
Observe that in special foliated coordinates for $\chi$ around $(x_0,t_0)\in M\times[0,+\infty)$, such that $g_f$ takes a diagonal expression with eigenvalues $\lambda_1,\dots, \lambda_n$, one has:
$$
\det(\chi)=1,\quad\det(g_f^T)=\lambda_1\cdots\lambda_n,\quad \sigma_f=g_f^{\bar j k}\chi_{j\bar k}=\sum_{j=1}^n\frac1{\lambda_j},\quad \gamma_f=\chi^{\bar j k}(g_f)_{j\bar k}=\sum_{j=1}^n\lambda_j,
$$
$$
 g_f^{\bar qp}g_f^{\bar ps}(g_f)_{s\bar q}=\sum_{j=1}^n\frac1{\lambda_j}=\sigma_f,\quad
g_f^{\bar qp}g_f^{\bar ps}\chi_{s\bar q}=\sum_{j=1}^n\frac1{\lambda_j^2}.
$$
\end{rem}
%Notice also that if we normalize $\chi$ in order to get $c=1$ and set special coordinates around a point $(x_0,t_0)\in M\times[0,\infty)$ such that $g_f$ takes a diagonal expression with eigenvalues $\lambda_1,\dots, \lambda_n$, then from $\frac 12d\eta_f-\chi>0$ one gets $\lambda_j\geq 1$. 
%\begin{lemma}\label{fund}
%Let $\lambda_j$, $j=1,\dots, n$ be positive real numbers. If $\lambda_j>1$ 
%%and for $n=2$:
%%$$
%% \lambda_1,\lambda_2\leq \frac{1+\epsilon}{1-\sqrt{1-\epsilon}},
%%$$
%then, for any $n> 2$:
%$$
%(n-1)\sum_{j=1}^n\frac1{\lambda_j^2}-2\sum_{j=1}^n\frac1{\lambda_j}+1\geq 0.
%$$
%\end{lemma}
%\begin{proof}
%By $\lambda_j\geq 1$, we get:
%\begin{equation}
%\begin{split}
%(n-1)\sum_{j=1}^n\frac1{\lambda_j^2}-2\sum_{j=1}^n\frac1{\lambda_j}+1=&\,\sum_{j=1}^n\left(\frac1{\sqrt{n-1}}-\frac{\sqrt{n-1}}{\lambda_j}\right)^2-\frac{1}{n-1}\\
%\geq&\,\sum_{j=1}^n\frac {(2-n)^2}{{n-1}}-\frac{1}{n-1}\\
%=&\frac {n(2-n)^2-1}{{n-1}}\\
%\geq&\,0.
%\end{split}\nonumber
%\end{equation}

\section{Second order estimates}\label{secord}

%Notice also that conditions \eqref{coord} depend only $(\xi,J)$ and therefore they hold for every Sasakian structure compatible with $(\xi,J)$. 

%a Sasakian structure can be regarded as a collection of K\"ahler structures each one defined on an open set of $\mathds{C}^n$.  

In order to develope the second order estimates we begin with the following lemma:

\begin{lemma}\label{log}
Let $(M,g,\xi,\Phi,\eta)$ be a $(2n+1)$--dimensional Sasakian manifold and let $f$ be a solution to \eqref{jflow} in $[0,+\infty)$. Then at any point $(x_0,t_0)\in M\times [0,+\infty)$:
$$
(\tilde \Delta_f-\partial_t)(\log \gamma_f)\geq\frac1{\gamma_f}g_f^{\bar qp}g_f^{\bar ps}R(\chi)^{\bar ba}_{\ \ s\bar q}(g_f)_{a\bar b}- \frac{1}{\gamma_f}g_f^{\bar kj}{\rm Ric}(\chi)_{j\bar k}.
$$
\end{lemma}
\begin{proof}
Compute:
$$
\partial_t\log(\gamma_f)=\frac{1}{\gamma_f}\chi^{\bar j k}\partial_t\left[(g_f)_{k\bar j}\right]=\frac{1}{\gamma_f}\chi^{\bar j k}\dot f_{,k\bar j},
$$
where:
\begin{equation}\label{dotfab}
\begin{split}
\dot f_{,a\bar b}=&-2g_f^{\bar k s}(g_f)_{\bar rs,\bar b}g_f^{\bar r p}(g_f)_{p\bar q,a}g_f^{\bar q j}\chi_{j\bar k}+g_f^{\bar qp}g_f^{\bar rs}\chi_{p\bar r}(g_f)_{a\bar b,s\bar q}\\
&+g_f^{\bar k p}(g_f)_{p\bar q,a}g_f^{\bar q j}\chi_{j\bar k,\bar b}+g_f^{\bar k s}(g_f)_{\bar rs,\bar b}g_f^{\bar r j}\chi_{j\bar k,a}-g_f^{\bar k j}\chi_{j\bar k,a\bar b}.
\end{split}\nonumber
\end{equation}
Thus:
\begin{equation}
\begin{split}
\partial_t\log(\gamma_f)=&\frac{1}{\gamma_f}\chi^{\bar b a}\left(-2g_f^{\bar k s}(g_f)_{\bar rs,\bar b}g_f^{\bar r p}(g_f)_{p\bar q,a}g_f^{\bar q j}\chi_{j\bar k}+g_f^{\bar qp}g_f^{\bar rs}\chi_{p\bar r}(g_f)_{a\bar b,s\bar q}\right.\\
&\left.+g_f^{\bar k p}(g_f)_{p\bar q,a}g_f^{\bar q j}\chi_{j\bar k,\bar b}+g_f^{\bar k s}(g_f)_{\bar rs,\bar b}g_f^{\bar r j}\chi_{j\bar k,a}-g_f^{\bar k j}\chi_{j\bar k,a\bar b}\right).%\\
%=&\frac{1}{\lambda_1+\lambda_2}\left(-2\frac1{\lambda_q\lambda_j\lambda_p}|(g_f)_{\bar pq,\bar a}|^2+\tilde \Delta_f[(g_f)_{a\bar a}]-\frac1{\lambda_j}{\rm Ric}(\chi)_{j\bar j}\right)
\end{split}\nonumber
\end{equation}
Taking special coordinates around $(x_0,t_0)$ we get:
$$
\partial_t\log(\gamma_f)=\frac{1}{\gamma_f}\chi^{\bar b a}\left(-2g_f^{\bar k s}(g_f)_{\bar rs,\bar b}g_f^{\bar r p}(g_f)_{p\bar q,a}g_f^{\bar q j}\chi_{j\bar k}+g_f^{\bar qp}g_f^{\bar rs}\chi_{p\bar r}(g_f)_{a\bar b,s\bar q}\right)+\frac1{\gamma_f}g_f^{\bar k j}{\rm Ric}(\chi)_{j\bar k}.
$$
Further:
$$
\tilde \Delta_f[\log\gamma_f]=g_f^{\bar qp}g_f^{\bar rs}\chi_{p\bar r}(\log\gamma_f)_{,s\bar q}=g_f^{\bar qp}g_f^{\bar rs}\chi_{p\bar r}(-\gamma_f^{-2}(\gamma_f)_{,\bar q}(\gamma_f)_{,s}+\gamma_f^{-1}(\gamma_f)_{,s\bar q}),
$$
where at $(x_0,t_0)$:
$$
(\gamma_f)_{,q}=-\chi^{\bar b r}\chi_{\bar rs,q}\chi^{\bar s a}(g_f)_{a\bar b}+\chi^{\bar ba}(g_f)_{a\bar b,q}=\chi^{\bar ba}(g_f)_{a\bar b,q},
$$
$$
(\gamma_f)_{,s\bar q}=-\chi^{\bar b r}\chi_{\bar rj,s\bar q}\chi^{\bar j a}(g_f)_{a\bar b}+\chi^{\bar ba}(g_f)_{a\bar b,s\bar q}.
$$
Thus:
$$
\tilde \Delta_f[\log\gamma_f]=-\frac1{\gamma_f^{2}}g_f^{\bar qp}g_f^{\bar ps}(g_f)_{a\bar a,\bar q}(g_f)_{b\bar b,s}+\frac1{\gamma_f}g_f^{\bar qp}g_f^{\bar ps}R(\chi)^{\bar ba}_{\ \ s\bar q}(g_f)_{a\bar b}+\frac1{\gamma_f}g_f^{\bar qp}g_f^{\bar ps}\chi^{\bar ba}(g_f)_{a\bar b,s\bar q},
$$
which implies:
\begin{equation}\label{estim1}
\begin{split}
(\tilde \Delta_f-\partial_t)(\log \gamma_f)&\geq\frac1{\gamma_f}g_f^{\bar qp}g_f^{\bar ps}R(\chi)^{\bar ba}_{\ \ s\bar q}(g_f)_{a\bar b}- \frac{1}{\gamma_f}g_f^{\bar kj}{\rm Ric}(\chi)_{j\bar k},
\end{split}
\end{equation}
where we used that by \cite[Lemma 3.2]{wein1}, one has:
$$
\gamma_f\chi^{b\bar a}g_f^{\bar k s}g_f^{\bar r p}g_f^{\bar q j}\chi_{j\bar k}(g_f)_{\bar rs,\bar b}(g_f)_{p\bar q,a}\geq\chi^{\bar ba}\chi^{\bar kj} g_f^{\bar qp}g_f^{\bar rs}\chi_{p\bar r}(g_f)_{a\bar b,\bar q}(g_f)_{j\bar k,s}.
$$
%By the assumption $\frac{c}{2}d\eta_f>\chi$, we get a uniform upper bound for $g_f^{\bar jk}$. Thus there exists a positive constant $B_0$ such that:
%$$
%g_f^{\bar kj}{\rm Ric}(\chi)_{j\bar k}\leq  B_0.
%$$
%Finally, let $B_1$ be a positive constant such that:
%$$
% R(\chi)^{\bar b a}_{\ \ j\bar k}\geq -B_1\chi^{b\bar a}\chi_{s\bar q}>-B_1\chi^{b\bar a}(g_f)_{s\bar q}.
%$$
%By Remark \ref{lambda}, plugging these last two estimates into \eqref{estim1} we finally get:
%$$
%(\tilde \Delta-\partial_t)(\log \gamma_f-Af)\geq -\frac{B_0}{\gamma_f}-B_1\sigma_f,
%$$
%as wished.
\end{proof}

Recall now that a uniform lower bound on the second derivatives of a solution $f$ to the Sasaki $J$-flow is obtained in \cite[Lemma 6.1]{VZ}.
%
%We will use here \cite[Prop. 5.1]{VZ}, which in our context reads as follows:
%\begin{lemma}\label{sup}
%Assume that $h\in C^\infty_B(M\times [0,T),\mathds{R})$ satisfies:
%$$
%(\partial_t-\tilde \Delta)h(x,t)\geq 0.
%$$
%Then:
%$$
%\sup_{(x,t)\in M\times [0,T)}h(x,t)\leq \sup_{x\in M} h(x,0).
%$$
%\end{lemma}
In order to get a uniform upper bound, we start proving the following proposition, which follows essentially \cite[Th. 2.1]{wein2} (see also \cite[Th. 3.1]{wein1} for the case $n=2$).

%%%%%%%%%%%%%%%%
%%%%%%%%%%%%%%%%

%
\begin{prop}\label{secondorder}
Let $(M,g,\xi,\Phi,\eta)$ be a $2n+1$-dimensional Sasakian manifold and let $f$ be a solution to \eqref{jflow} in $[0,+\infty)$. Assume that $\frac{c}2d\eta-(n-1)\chi>0$. Then, for any $t\geq0$ there exist constants $A$ and $C$, depending only on the initial data, such that $\gamma_f\leq  Ce^{A(f-\inf_{M\times [t,0]} f)}$ in $[0,t]$.
\end{prop}
\begin{proof}
Normalize $\chi$ in order to get $c=1$ (i.e. $\frac{1}2d\eta-(n-1)\chi>0$).
Fix $t> 0$ and let $(x_0,t_0)$ be a maximum in $M\times [0,t]$ for $\log \gamma_f-Af$, where $A$ is a constant to be fix later.
%Observe that from $\frac{1}2d\eta_0-\chi>0$ and the lower bound on the second derivatives of $f$ (see ), it follows $\frac{1}2d\eta_f-\chi>0$. Thus, we can apply 
By Lemma \ref{log} above at $(x_0,t_0)$ we get:
$$
(\tilde \Delta_f-\partial_t)(\log \gamma_f)\geq\frac1{\gamma_f}g_f^{\bar qp}g_f^{\bar ps}R(\chi)^{\bar ba}_{\ \ s\bar q}(g_f)_{a\bar b}- \frac{1}{\gamma_f}g_f^{\bar kj}{\rm Ric}(\chi)_{j\bar k}.
$$
Further, we have:
\begin{equation}\label{estimf}
(\tilde \Delta_f-\partial_t)f=g_f^{\bar qp}g_f^{\bar rs}\chi_{p\bar r}f_{s\bar q}-\dot f=g_f^{\bar qp}g_f^{\bar ps}(g_f)_{s\bar q}-g_f^{\bar qp}g_f^{\bar ps}(g_0)_{s\bar q}-\dot f.
\end{equation}
Thus, by Remark \ref{lambda} and 
since with our normalization a solution $f$ to \eqref{jflow} satisfies $\dot f=1-\sigma_f$, it follows:
%\begin{equation}\label{estim2}
%\begin{split}
$$
(\tilde \Delta_f-\partial_t)(\log \gamma_f-Af)\geq  \frac1{\gamma_f}g_f^{\bar qp}g_f^{\bar ps}R(\chi)^{\bar ba}_{\ \ s\bar q}(g_f)_{a\bar b}- \frac{1}{\gamma_f}g_f^{\bar kj}{\rm Ric}(\chi)_{j\bar k}-2A\sigma_f+Ag_f^{\bar qp}g_f^{\bar ps}(g_0)_{s\bar q}+A.
$$
%\end{split}
%\end{equation}
Let $C_0$ be a positive constant such that:
$$
 R(\chi)^{\bar b a}_{\ \ j\bar k}\geq -C_0\chi^{b\bar a}(g_0)_{s\bar q},%>-C_0\chi^{b\bar a}(g_0)_{s\bar q},
$$
From $\frac12d\eta-\chi>0$ it follows that  we can choose $\epsilon>0$ small enough to have:
\begin{equation}\label{epsilon}
\frac12d\eta\geq (n-1+(n+1)\epsilon)\chi,
\end{equation}
and we can set $A$ big enough such that:
$$
\epsilon Ag_f^{\bar qp}g_f^{\bar ps}(g_0)_{s\bar q}\geq -C_0g_f^{\bar qp}g_f^{\bar ps}(g_0)_{s\bar q}-\frac{1}{\gamma_f}g_f^{\bar kj}{\rm Ric}(\chi)_{j\bar k}.
$$
Thus:
$$
(\tilde \Delta_f-\partial_t)(\log \gamma_f-Af)\geq A\left((1-\epsilon)g_f^{\bar qp}g_f^{\bar ps}(g_0)_{s\bar q}-2\sigma_f+1\right).
$$
Since at $(x_0,t_0)$ it follows easily by \eqref{epsilon} that one has:
$$
(1-\epsilon)(g_0)_{s\bar q}\geq  (n-1+\epsilon)\chi_{s\bar q},
$$
by Remark \ref{lambda} we finally get:
$$
(\tilde \Delta-\partial_t)(\log \gamma_f-Af)\geq A\left((n-1+\epsilon)\sum_{j=1}^n\frac1{\lambda_j^2}-2\sum_{j=1}^n\frac1{\lambda_j}+1\right).
$$
At this point, observe that $(x_0,t_0)$ has been chosen to be a global maximum in $M\times [0,t]$ and thus (see Remark \ref{delta}):
$$
0\geq(\tilde \Delta-\partial_t)(\log \gamma_f-Af),
$$
which implies:
\begin{equation}%\label{estim3}
(n-1+\epsilon)\sum_{j=1}^n\frac1{\lambda_j^2}-2\sum_{j=1}^n\frac1{\lambda_j}+1\leq 0.\nonumber
\end{equation}
This last inequality implies an upper bound for all $\lambda_j$, as it follows considering that we can rewrite it as:
$$
\sum_{j=1}^n\left(\frac{1}{\sqrt{n-1+\epsilon}}-\frac{\sqrt{n-1+\epsilon}}{\lambda_j}\right)^2-\frac n{n-1+\epsilon}+1\leq 0,
$$
an thus for any $j=1,\dots, n$:
$$
\frac{1}{\sqrt{n-1+\epsilon}}-\frac{\sqrt{n-1+\epsilon}}{\lambda_j}\leq \frac {\sqrt {1-\epsilon}}{\sqrt{n-1+\epsilon}},
$$
i.e.:
$$
\lambda_j\leq\frac {{n-1+\epsilon}}{1-\sqrt {1-\epsilon}}.
$$
It follows that at $(x_0,t_0)$, $\gamma_f$ is bounded above. Since $(x_0,t_0)$ is the global maximum in $[0,t]$ for $\log\gamma_t-Af$, we get that
$$
\log\gamma_t-Af\leq \log C-A\inf_{M\times [t,0]} f,
$$ 
i.e.:
$$
\gamma_t\leq Ce^{A(f-\inf_{M\times [0,t]} f)},
$$
as wished.
\end{proof}

\section{$C^0$ estimates and the proof of Theorem \ref{main}}\label{c0}

In order to get a uniform upper bound for a solution to \eqref{jflow} we modify the arguments in \cite{wein1, wein2}. 
For this purpose, let $g_\chi$ be the Riemannian metric which has $\chi$ as transverse K\"ahler metric, i.e.:
$$
g_\chi(\cdot,\cdot)=\chi(\cdot,\Phi\cdot)+\eta(\cdot)\eta(\cdot).
$$
Observe that since $(M,g_\chi)$ is a compact Riemannian manifold, there exists a Green function $G(x,y)$ which satisfies for any $u\in C^\infty (M)$:
$$
u(x)=\int_M G(x,y)\Delta u(y)d\mu(y)+\frac{1}{\int_Md\mu}\int_Mu\,d\mu,
$$
where $d\mu$ and $\Delta$ are respectively the volume form and the Riemannian Laplacian associated to ${g_\chi}$.
By \cite[Prop. 2.8]{nitta} $\Delta_\chi\psi=-\Delta\psi$ for any $\psi\in C^\infty_B(M, \mathds{R})$, where $\Delta_\chi$ is the basic Laplacian associated to $\chi$, i.e. it is locally expressed by:
$$
\Delta_\chi\psi=\,\chi^{\bar jr}\psi_{,r\bar j}\,,\quad \mbox{ for }\psi\in C^{\infty}_B(M, \mathds{R})\,,
$$
(in our notation the basic Laplacian has the opposite sign of \cite{nitta} one).
Thus, for any $\psi\in C^\infty_B(M, \mathds{R})$ we have:
\begin{equation}\label{green}
\psi(x)=-\int_M G(x,y)\Delta_\chi\psi(y)d\mu+\frac{1}{\int_Md\mu}\int_M\psi d\mu.
\end{equation}
\begin{rem}\label{decreasing}\rm
Notice that $\Delta_\chi f$ is uniformly bounded from below, as it follows easily from the definition of $\Delta_\chi f$ and by observing that:
$$
\chi^{\bar j k}(g^T_f)_{j\bar k}=\chi^{\bar j k}((g^T)_{j\bar k}+f_{,j\bar k})>0.
$$
\end{rem}

\begin{prop}\label{estimates}
Let $(M,g,\xi,\Phi,\eta)$ be a $(2n+1)$-dimensional Sasakian manifold and let $f$ be a solution to \eqref{jflow}. Then there exist two positive constants $C_0$ and $C_1$, depending only on the initial data, such that:
$$0\leq \sup_M f\leq C_0-C_1\inf_Mf.$$ 
\end{prop}
\begin{proof}
%By Lemma \ref{lb} above, there exists a uniform lower bound for $\inf_M f$ and thus, since $M$ is compact, for $f$. 
Observe first that from $f\in \mathcal H_0$ by \eqref{h0} we get:
\begin{equation}\label{if0}
\sum_{p=0}^n\frac{n!}{(p+1)!(n-p)!}\int_Mf\,\eta\wedge(d\eta)^{n-p}\wedge (i\partial_B\bar\partial_Bf)^p=0.
\end{equation}
Thus $f$ vanishes somewhere and we have $\sup_M f\geq 0$. 

In order to prove the second inequality, let $B_0$, $B_1$ be constants such that:
$$
d\mu\leq B_0\,\eta\wedge  (d\eta)^{n}, \quad d\eta\leq B_1\chi.
$$
From \eqref{if0} we get:
$$
\int_Mf\eta\wedge (d\eta)^n=-n\int_M\eta\wedge (d\eta_f)^{n-1}\wedge(i\partial_B\bar\partial_B f)=-n\int_M\eta\wedge (d\eta_f)^{n}+n\int_M\eta\wedge d\eta\wedge (d\eta_f)^{n-1},
$$
and thus:
%Thus, for any solution $f$ to the $J$-flow, we have:
%$$
%\partial_t J_\chi(f)=\frac12\int_M\dot f\,\eta\wedge\chi\wedge d\eta_f\leq 0.
%$$
%Recalling that 
%$$
%\int_M\left(f-\inf_M f\right)\eta_0\wedge\chi\wedge (i\partial_B\bar\partial_B f)\geq 0,
%$$
%and
%$$
%\int_M\,\eta_0\wedge\chi\wedge d\eta_0=\int_M\,\eta_0\wedge\chi\wedge d\eta_f,
%$$
 \begin{equation}
\begin{split}
\int_Mf\,d\mu\leq &\,B_0\int_Mf\,\eta \wedge (d\eta)^{n}\\
=&-nB_0\int_Mf\eta \wedge (d\eta_f)^n+nB_0\int_f\eta\wedge d\eta\wedge (d\eta_f)^{n-1}\\
\leq&\, -nB_0\int_Mf\eta \wedge (d\eta_f)^n+nB_0B_1\int_f\eta\wedge \chi\wedge (d\eta_f)^{n-1}\\
=&\,-nB_0\int_M(f-\inf_Mf)\,\eta \wedge (d\eta_f)^n-nB_0\inf_Mf\int_M\eta \wedge (d\eta)^n+nB_0B_1\int_f\eta\wedge \chi\wedge (d\eta)^{n-1}\\
\leq&\,nB_0B_1\int_f\eta\wedge \chi\wedge (d\eta_f)^{n-1}-nB_0\int_M\eta \wedge (d\eta)^n\inf_Mf.
\end{split}\nonumber
\end{equation}
%Observe that by Lemma \ref{decreasing} above, there exists a constant $C_1$ depending only on the initial data, such that:
%$$
%\int_Mf\,\eta_0\wedge\chi\wedge (d\eta_f)^{n-1}\leq C_1.
%$$
%which is a uniform upper bound for $\int_Mfd\mu$ depending on the existence of a uniform lower bound for $\inf_M f$.
Thus, by \eqref{green}:
$$
f(x)\leq -\int_M G(x,y)\Delta_\chi f(y)\,\eta\wedge (d\eta)^n+nB_0B_1\int_f\eta\wedge \chi\wedge (d\eta)^{n-1}-nB_0\int_M\eta \wedge (d\eta)^n\inf_Mf,
$$
and conclusion follows by the existence of a lower bound for the Green function of $g$ and from Remark \ref{decreasing}.
\end{proof}

It remains to prove that $\inf_M f$ is uniformly bounded from above. Following \cite{wein2}, assume that such bound does not exist. Then there exists a sequence of time $t_i$ such that $t_i\rightarrow \infty$ implies $\inf_{t_i}\inf_Mf\rightarrow \infty$.
Fix $i$ and set $\psi_i(x)=f(x,t_i)-\sup_Mf(x,t_i)$. Since by Prop. \ref{estimates} above $\sup_Mf\geq0$, we have $\sup_M\psi_i=0$. This last fact, together with Prop. \ref{secondorder} in the previous section, will lead us with Prop. \ref{normac0} to the contradiction $||e^{-B\psi_i}||_{C^0}< 1$, where $B=A/(4-\epsilon)$ for a small $\epsilon>0$ which is set in the proof of Prop. \ref{normac0}. 

%%%%%%%%%%%%%%%%%%%%%%%%%%%%%%%
%%%%%%%%%%%%%%%%%%%%%%%%%%%%%%%
%%%%%%%%%%%%%%%%%%%%%%%%%%%%%%%
%%%%%%%%%%%%%%%%%%%%%%%%%%%%%%%
%%%%%%%%%%%%%%%%%%%%%%%%%%%%%%%
%%%%%%%%%%%%%%%%%%%%%%%%%%%%%%%
%%%%%%%%%%%%%%%%%%%%%%%%%%%%%%%
%%%%%%%%%%%%%%%%%%%%%%%%%%%%%%%
We begin with the following lemma.

\begin{lemma}\label{stimanabla1}
Let $(M,g,\xi,\Phi,\eta)$ be a compact Sasaki $2n+1$-dimensional manifold and let $\chi$ a transverse K\"ahler form on $M$ and $a>0$. If $\psi\in \mathcal H$ satisfies $\gamma_\psi\leq Ce^{A(\psi-\inf_{M\times [t,0]} \psi)}$ for some constants $A$ and $C$, then:
$$
\int_M|\nabla e^{-a\psi}|^2\eta\wedge \chi^n\leq \frac{aC}2\,e^{-A\inf_{M\times [t,0]} \psi}\int_Me^{(A-2a)\psi} \,\eta\wedge \chi^n.
$$ 
\end{lemma}

\begin{proof}
Observe first that:
\begin{equation}\label{nablapsi}
\int_M|\nabla e^{-a\psi}|^2\eta\wedge \chi^n=\,(\partial_Be^{-a\psi},\partial_Be^{-a\psi})_\chi=(\partial_B^*\partial_Be^{-a\psi},e^{-a\psi})_\chi=-(\Delta_\chi e^{-a\psi},e^{-a\psi})_\chi,
\end{equation}
From:
$$
\Delta_\chi e^{-a\psi}=\chi^{\bar kj}\left(e^{-a\psi}\right)_{,j\bar k}=\chi^{\bar kj}\left(-ae^{-a\psi}\psi_{,j\bar k}+a^2e^{-a\psi}\psi_{,j}\psi_{,\bar k}\right),
$$
it follows that:
$$
(\Delta_\chi e^{-a\psi},e^{-a\psi})_\chi=-a(\Delta_\chi \psi,e^{-2a\psi})_\chi+a^2(\partial_B\psi,e^{-2a\psi}\partial_B\psi)_\chi,
$$
and since
$$
-2a\,e^{-2a\psi}\partial_B\psi=\partial_Be^{-2a\psi},
$$
we get:
$$
(\Delta_\chi e^{-a\psi},e^{-a\psi})_\chi=-a(\Delta_\chi \psi,e^{-2a\psi})_\chi-\frac a2(\partial_B\psi,\partial_B e^{-2a\psi})_\chi=-\frac a2(\Delta_\chi \psi,e^{-2a\psi})_\chi.
$$
Thus, plugging this last equality into \eqref{nablapsi} we get:
$$
\int_M|\nabla e^{-a\psi}|^2\eta\wedge \chi^2=\frac a2(\Delta_\chi \psi,e^{-2a\psi})_\chi.
$$
Since $\chi^{j\bar k}(g_0^T)_{j\bar k}\geq 0$ and we assumed $\gamma_\psi\leq Ce^{A(\psi-\inf_{M\times [0,t]} \psi)}$, conclusion follows by observing that:
$$
\Delta_\chi \psi=\gamma_\psi-\chi^{j\bar k}(g_0^T)_{j\bar k}.
$$
%$$
%\frac a2(Ce^{A(\psi-\inf_{M\times [t,0]} \psi)},e^{-2a\psi})_\chi=\frac a2(Ce^{-A\inf_{M\times [t,0]} \psi},e^{A-2a\psi})_\chi
%$$
\end{proof}
The next lemma is a Sasakian version of \cite[Prop. 2.1]{tian}.
\begin{lemma}\label{boundint}
Let $(M,g,\xi,\Phi,\eta)$ be a compact Sasaki $2n+1$-dimensional manifold and let $\chi$ be a second transverse K\"ahler form on $M$. Assume that $\psi\in \mathcal H$ satisfies $\sup_M\psi=0$. Then for a small enough $\alpha>0$, there exists a constant $C'$ depending only on the initial data, such that:
$$
\int_Me^{-\alpha\psi}\eta\wedge\chi^n\leq C'.
$$
\end{lemma}
\begin{proof}
By \eqref{green} we have:
$$
0=\sup_M\psi(x)\leq\frac{1}{\int_Md\mu}\int_M\psi d\mu+\sup_M\int_M G(x,y)(-\Delta_\chi\psi(y))d\mu,
$$
which implies:
$$
\frac{1}{\int_Md\mu}\int_M\psi d\mu\geq \sup_M \int_M G(x,y)(-\Delta_\chi\psi(y))d\mu\geq- B_1,
$$
i.e., for some constant $B_2$:
$$
\int_M\psi\, \eta\wedge\chi^n\geq- B_2.
$$
At this point, recall that locally $M$ can be described by special coordinates $(z_1,\dots, z_n,z)\in\mathds{C}^n\times \mathds R$ and being $\psi$ basic, it is constant in $z$. Let $\{B_{r_i}\}_i$ be geodesic balls that cover $M$. For each $i$, set special coordinates on $B_{r_i}$. Since the support of the smooth function $\psi'=\psi-\psi(0)$ is contained in $\tilde B_{r_i}=\left\{(z_1,\dots, z_n,z)\in B_{r_i}|\, z=0\right\}$, we can apply the same method as in \cite[Prop. 2.1]{tian} to get the existence of a constant $B_3$ such that:
%if $B_{r_i}$ are geodesic balls that cover $M$
$$
\int_Me^{-\alpha\psi'}\eta\wedge\chi^n\leq B_3.
$$
Then we have:
$$
\int_Me^{-\alpha\psi}\eta\wedge\chi^n\leq B_3e^{-\alpha\psi(0)},
$$
and the assertion follows setting $C'=B_3e^{-\alpha\psi(0)}$.
\end{proof}
We recall here a Sobolev inequality for a compact Riemannian manifold $(M,g)$ needed in the proof of Prop. \ref{normac0} below. For any smooth function $h$ on $M$ and $a>0$ denote:
$$
||h||_a=\left(\int_M|h|^ad\mu_g\right)^{\frac1a},
$$
where $d\mu_g$ is the volume form associated to $g$. Then we have the following (see e.g. \cite[Th. 2.6]{hebey}):
\begin{theorem}[Sobolev inequality]
Let $(M,g)$ be a smooth compact Riemannian manifold of dimension $m$. Then for any real numbers $1\leq q<p$ with $1/p=1/q-1/m$, there exists a constant $C_1$ such that:
$$
||h||_p\leq C_1\left(||\nabla h||_q+||h||_q\right).
$$
\end{theorem}
In particular, setting $q=2$, $m=2n+1$ and raising to the second power, there exists a constant $C_1$ such that:
\begin{equation}\label{sobolev}
||h||^2_b\leq C_1\left(||\nabla h||^2_2+||h||^2_2\right), 
\end{equation}
for $b=2(2n+1)/(2n-1)$.
\begin{prop}\label{normac0}
Let $(M,g,\xi,\Phi,\eta)$ be a compact Sasaki $2n+1$-dimensional manifold and let $\chi$ be a second transverse K\"ahler form on $M$. If $\psi\in \mathcal H$ satisfies:
\begin{enumerate}
\item[($i$)] $\sup_M\psi=0$,
\item[($ii$)] $\gamma_\psi\leq Ce^{A(\psi-\inf_{M\times [t,0]} \psi)}$,
\end{enumerate}
then for some small $\epsilon>0$, $||e^{-\frac{A}{4-\epsilon}\psi}||_{C^0}\leq C'$, where $C'$ is a constant depending on $A$, $C$ and the initial data.
\end{prop}
\begin{proof}
Since $(M,g)$ is a Riemannian manifold, then by \eqref{sobolev}, for $b=2(2n+1)/(2n-1)$ and some constant $C_1$, the following Sobolev inequality holds for any $u\in C^\infty(M)$:
$$
||u||^2_{b}\leq C_1\left(||\nabla u||^2_2+||u||^2_2\right),
$$
i.e., since the Sobolev inequality is independent from the volume form chosen:
$$
\left(\int_M|u|^{b}\eta\wedge \chi^n\right)^{\frac2b}\leq C_1\left(\int_M|\nabla u|^2\eta\wedge \chi^n+\int_M|u|^2\eta\wedge \chi^n\right).
$$
For $0<q<p$, set $u=e^{-\frac A{3p+q}\psi}$. Since $\int_M|u|^2\eta\wedge \chi^n\geq 0$, by Lemma \ref{stimanabla1} applied to $u^{2p}$, there exists a constant $C_1$ such that:
%$$
%\int_M|\nabla e^{-a\psi}|^2\eta\wedge \chi^n\leq \frac{aC}2\,e^{-A\inf_{M\times [t,0]} \psi}\int_Me^{(A-2a)\psi} \,\eta\wedge \chi^n.
%$$ 
\begin{equation}
\begin{split}
\left(\int_M|e^{-\frac{2pA}{3p+q}\psi}|^{b}\eta\wedge \chi^n\right)^{\frac2b}\leq &\,C_1\left(\int_M|\nabla e^{-\frac{2pA}{3p+q}\psi}|^2\eta\wedge \chi^n+\int_M|e^{-\frac{2pA}{3p+q}\psi}|^2\eta\wedge \chi^n\right)\\
\leq&\,\frac{C_1\:C\:p\:A}{3p+q}\,e^{-A\inf_{M\times [t,0]} \psi}\int_Me^{-\frac{A}{3p+q}(p-q)\psi} \,\eta\wedge \chi^n
%\leq &\,\frac{A\,C\,C_1\,p}{4q}\,e^{-A\inf_{M\times [t,0]} \psi}\int_Me^{-A\frac{p-q}{q}\psi} \,\eta\wedge \chi^n.
\end{split}\nonumber
\end{equation}
Then, raising to the power of $1/4p$, we get
\begin{equation}\label{ubp}
||u||_{2bp}\leq \left(\frac{C_1\:C\:p\:A}{3p+q}\right)^{\frac1{4p}}\,e^{-\frac{A}{4p}\inf_{M\times [t,0]} \psi}||u||^{\frac{p-q}{4p}}_{p-q}%C_1^{\frac{2\gamma}A}(1+e^{1-\gamma})^{\frac{2\gamma}A}||e^{-\psi}||^{2}_{A/\gamma}
\end{equation}
Now set:
$$
s_0=p,\quad s_k=p\,(2b)^k+q\left((2b)^{k-1}+\dots+2b+1\right)\ \ \ \textrm{for}\ k=1,2,\dots.
$$ 
With this notation \eqref{ubp} reads:
\begin{equation}\label{ubp2}
||u||_{2bs_0}\leq \left(C_1\:C\:A\right)^{\frac1{4s_0}}\left(\frac{s_0}{3s_0+q}\right)^{\frac1{4s_0}}\,e^{-\frac{A}{4s_0}\inf_{M\times [t,0]} \psi}||u||^{\frac{p-q}{4p}}_{p-q}.
\end{equation}
Replacing $p$ with $2bp+q$, $s_k$ change in $s_{k+1}$ for all $k=0,1,2,\dots$ and we get:
$$
||u||_{2bs_1}\leq \left(C_1\:C\:A\right)^{\frac1{4s_1}}\left(\frac{s_1}{3s_1+q}\right)^{\frac1{4s_1}}\,e^{-\frac{A}{4s_1}\inf_{M\times [t,0]} \psi}\left(||u||_{2bs_0}\right)^{\frac{bs_0}{2s_1}},
$$
and thus by \eqref{ubp2}:
$$%\begin{equation}\label{ubp3}
||u||_{2bs_1}\leq  \left(C_1\:C\:A\right)^{\frac1{4s_1}\left(1+\frac b2\right)}\left(\frac{s_1}{3s_1+q}\right)^{\frac1{4s_1}}\left(\frac{s_0}{3s_0+q}\right)^{\frac b2\frac1{4s_1}}\,e^{-\frac{A}{4s_1}\left(1+\frac b2\right)\inf_{M\times [t,0]} \psi}\left(||u||^{\frac{p-q}{4p}}_{p-q}\right)^{\frac{bs_0}{2s_1}},
$$%\end{equation}

If we iterate this procedure $k$ times, we get:
$$
||u||_{2bs_k}\leq \left(C_1\:C\:A\right)^{a_k}\left(\prod_{j=0}^k\left(\frac{s_j}{3s_j+q}\right)^{\left(\frac b2\right)^{k-j}}\right)^{\frac1{4s_k}}e^{-Aa_k\inf_{M\times [0,t]} \psi}\left(||u||^{\frac{p-q}{4p}}_{p-q}\right)^{\frac{b^ks_0}{2^ks_k}},
 $$
 where we set $a_k=\frac{1}{4s_k}\sum_{j=0}^{k}\left(\frac b2\right)^k$.
Observe now that, since $b=2(2n+1)/(2n-1)$ implies $b/2> 1$ and $s_k\rightarrow +\infty$ as $k$ approaches infinity, we have:
$$
\lim_{k\rightarrow +\infty}a_k=\lim_{k\rightarrow +\infty}\left(\frac{1}{4s_k}\sum_{j=0}^{k}\left(\frac b2\right)^k\right)=0,\quad \lim_{k\rightarrow +\infty}\left(\frac{b^ks_0}{2^ks_k}\right)=0;%=\frac{\frac{1}{2^k}+\frac{1}{b2^{k-1}}+\dots+\frac1{b^k}}{p+q(\frac1{b}+\frac1{b^{2}}+\dots+\frac{1}{b^{k}})},
$$
%$$
%\lim_{k\rightarrow +\infty}\left(\frac{b^ks_0}{2^ks_k}\right)=0;%=\frac{s_0}{2^k(p+q(\frac1b+\frac1{b^{2}}+\dots+\frac1{b^{k}}))},
%$$
further, since for any $j=0,\dots ,k$,
$$
\lim_{k\rightarrow +\infty}\left(\frac1{s_k}\left(\frac b2\right)^{k-j}\right)=0,\quad \frac{p}{3p+q}\leq\frac{s_j}{3s_j+q}\leq \frac13
$$
we have:
$$
\lim_{k\rightarrow +\infty}\left(\prod_{j=0}^k\left(\frac{s_j}{3s_j+q}\right)^{\left(\frac b2\right)^{k-j}}\right)^{\frac1{4s_k}}=1.
$$
Setting $p=1$ and $q=1-\epsilon$, by Lemma \ref{boundint} there exists a small enough $\epsilon>0$ such that:
$$
||u||_{p-q}=||u||_{\epsilon}= \int_Me^{-\frac{A\epsilon}{4-\epsilon}\psi}\eta\wedge\chi^n,
$$
is bounded and the bound on $||e^{-\frac{A}{4-\epsilon}\psi}||_{C^0}$ follows readly.
\end{proof}

\begin{cor}\label{boundf}
Let $\frac{c}2d\eta-(n-1)\chi>0$. Then the second order derivatives of a solution $f$ to the Sasaki $J$-flow are uniformly bounded from above.
\end{cor}
\begin{proof}
By Prop. \ref{normac0} and the discussion above, $\inf_Mf$ is uniformly bounded from above. The bound on the second order derivatives of $f$ follows then by Prop. \ref{secondorder} and Prop. \ref{estimates}.
\end{proof}

%\section{The proof of Theorem \ref{main}}
%
%\begin{lemma}\label{hypca}
%Let $(M,g,\xi,\Phi,\eta)$ be a $2n+1$ dimensional Sasaki manifold and let $f$ be a solution to \eqref{jflow}. Let $g_t$, $t\in[0,+\infty)$ be the family of Riemannian metric con $M$ such that:
%$$
%g_t(\cdot,\cdot)=g_f^T(\cdot,\cdot)+\eta(\cdot)\eta(\cdot).
%$$
%Then $g_t$ satisfies:
%\begin{enumerate}
%\item $C_1(g_0)_{ij}\leq (g_t)_{ij}\leq C_2 (g_0)_{ij}$,
%\item $\left|\partial_t (g_t)_{ij}\right|\leq C_3(g_0)_{ij}$,
%\item $({\rm Ric}_t)_{ij}\geq -K (g_0)_{ij}$,
%\end{enumerate}
%where $C_1$, $C_2$, $C_3$ and $K$ are positive constant independent of $t$.
%\end{lemma}
%\begin{proof}
%By \cite[Lemma 6.1]{VZ} and by Corollary \ref{boundf} above and since $f$ satisfies \eqref{jflow} and $\dot f$ satisfies the heat equation $\partial_t\dot f=-\tilde \Delta \dot f$, we have uniform bounds for $(g^T_f)_{j\bar k}$, $(\partial_t g^T_f)_{j\bar k}$, and all the covariant derivatives of $(g^T_f)_{j\bar k}$ and for $(\partial_t g^T_f)_{j\bar k}$.
%\end{proof}
%\begin{lemma}
%Let $(M,g,\xi,\Phi,\eta)$ be a $2n+1$ dimensional Sasaki manifold and let $f$ be a solution to \eqref{jflow}. Then
%$$
%\sup_M f(x, t_1)-\inf_M f(x,t_2)
%$$
%\end{lemma}
We are now in the position of proving Theorem \ref{main}.
\begin{proof}[Proof of Theorem \ref{main}]
By \cite[Th. 1.1]{VZ}, there exists a solution $f\!: M\times [0,+\infty)\rightarrow \mathds{R}$ to the Sasaki $J$-flow. By \cite[Lemma 6.1]{VZ} and by Corollary \ref{boundf} above, the second derivatives $\partial_j\partial_{\bar k}f$ are uniformly bounded. Since a solution to the Sasaki $J$-flow can be regarded as a solution to the K\"ahler $J$-flow on small open balls of $\mathds{C}^n$, by \cite[Th. 7.1]{VZ} we get uniform $C^\infty$ bounds on $f$. Then, by Ascoli-Arzel\`a Theorem, given a sequence $t_j\in [0,\infty)$, $t_j\rightarrow \infty$, there exists a subsequence $f_{t_j}$ converging in $C^\infty$-norm to a function $f_\infty$ as $t_j\rightarrow \infty$.

At this point, observe that $\dot f$ satisfies the heat equation $\partial_t\dot f=-\tilde \Delta \dot f$ and we have uniform bounds for $(g^T_f)_{j\bar k}$, $(\partial_t g^T_f)_{j\bar k}$, and all the covariant derivatives of $(g^T_f)_{j\bar k}$ and for $(\partial_t g^T_f)_{j\bar k}$.
%By \cite[Lemma 6.1]{VZ} and by Corollary \ref{boundf} above and since $f$ satisfies \eqref{jflow} and $\dot f$ satisfies the heat equation $\partial_t\dot f=-\tilde \Delta \dot f$, we have uniform bounds for $(g^T_f)_{j\bar k}$, $(\partial_t g^T_f)_{j\bar k}$, and all the covariant derivatives of $(g^T_f)_{j\bar k}$ and for $(\partial_t g^T_f)_{j\bar k}$.
Then we get uniform bounds also for the family $g_t$, $t\in[0,+\infty)$, of Riemannian metric on $M$ defined by:
$$
g_t(\cdot,\cdot)=g_f^T(\cdot,\cdot)+\eta(\cdot)\eta(\cdot),
$$
and for all its covariant derivatives. Thus we can apply the argument in \cite{ca} (Th. 2.1 and discussion below) to get:
$$
\sup_Mf-\inf_Mf\leq C_0e^{-C_1 t}.
$$
for some constant $C_0$ and $C_1$ independent of $t$. The convergence of $f$ in the $C^\infty$ topology follows by the same argument as in \cite[Sec. 5]{wein1}.
\end{proof}

\section{Mabuchi $K$-energy and the $J$-flow}
Let $\bar {\mathcal H}$ be the completion of $\mathcal H$ with respect to the $C^2_w$-norm. In \cite{guanzhangA}, P. Guan, X. Zhang proved that any two points in $\mathcal H$ can be connected by a $C^{1,1}$--geodesic. By definition, a $C^{1,1}$-geodesic is a curve in $\bar {\mathcal H}$ obtained as weak limit of solutions to:
$$
\left(\ddot f-\frac14 |d_B \dot f|^2_f\right)\,\eta\wedge (d\eta_f)^n=\epsilon\,\eta\wedge (d\eta)^n\,.
$$
This result allows us to prove the following (cfr. \cite[Prop. 3]{chenM}).
\begin{prop}\label{jbound}
If there exists a critical metric then the functional $J_\chi\!:\mathcal H_0\rightarrow \mathds R$ is uniformly bounded from below.
\end{prop}
\begin{proof}
Let $f_\infty$ be a critical point in $\mathcal H_0$, $f_1\in \mathcal H_0$ and let $f\!:[0,1]\rightarrow \bar{\mathcal H}$ be a $C^{1,1}$-geodesic such that $f(0)=f_\infty$, $f(1)=f_1$. Then by the estimates in the proof of Prop. 3.2 in \cite{VZ}, $J_\chi$ satisfies:
$$
\partial^2_t J_\chi(f)\geq 0.
$$
Further, being $f_\infty$ a critical point, $\partial_tJ_\chi(f)|_{t=0}=0$. Thus, we have $J_\chi(f_1)\geq J_\chi(f_\infty)$ for any $f_1\in \mathcal H_0$.
\end{proof}
Let $\alpha$ be a transverse K\"ahler form on $M$ and denote by $[\alpha]_B$ the basic $(1,1)$ class associated to $\alpha$.
Define:
$$
\mathcal K=\{ \textrm{transverse K\"ahler form in the basic } (1,1) \textrm{ class } [d\eta]_B\},
$$
and observe that $\mathcal H_0\simeq \mathcal K$. Further, denote by $s^T_f$ the transverse scalar curvature associated to $d\eta_f$, namely in local coordinates $s^T_f=(g_f^T)^{\bar kj}{\rm Ric}(d\eta_f)_{j\bar k}$, where ${\rm Ric}(d\eta_f)$ is the Ricci tensor associated to $d\eta_f$. Let us also denote by $\rho^T$ the transverse Ricci form associated to ${\rm Ric}(d\eta)$ and by $\bar s^T$ the average scalar curvature defined by:
$$
\bar s^T=\frac{2n\int_M \rho^T \wedge \eta\wedge (d\eta)^{n-1}}{\int_M \eta\wedge (d\eta)^n}\,. 
$$
The Mabuchi $K$-energy in the Sasakian context has been introduced by A. Futaki, H. Ono and G. Wang in \cite{F}. According to the notation in \cite{Z}, it is defined as follows.
Let $f_0,f_1\in  \mathcal{H}$ and $f\colon [0,1]\to\mathcal{H}
$ be a smooth path satisfying $f(0)=f_0$, $f(1)=f_1$. Then the functional $\mathcal M\!:\mathcal H\times \mathcal H\rightarrow R$ defined by:
$$
\mathcal M(f_0,f_1):=\int_0^1\int_M \dot f\,(s^T_f-\bar s^T)\wedge \eta\wedge (d\eta_f)^{n}\,dt,
$$
is well defined and factors through $\mathcal H_0\times H_0$ (see e.g. \cite[Lemma 3.2]{Z}). Define the $K$-energy map of the transverse K\"ahler class $[d\eta]_B$ by $
\mathcal M\!:\mathcal K\rightarrow R,$ $\mathcal M(d\eta_f)=\mathcal M(d\eta,d\eta_f).
$
Further, the map
$
\mathcal M\!:\mathcal H\rightarrow R,$ $\mathcal M(f)=\mathcal M(0,f),
$
is called the $K$-energy map of $\mathcal H$.
%In \cite[Prop. 3.2]{Z} the following relation between the Sasaki Mabuchi $K$-energy map and the $J_\chi$ functional (cfr. \cite{chenM} for a similar result in the K\"ahler context):

We can prove now our second result, which should be compared with \cite{chenM, wein1, wein2}.
\begin{theorem}\label{second}
Let $(M,\xi,\Phi,\eta,g)$ be a Sasakian manifold and assume that $-\rho^T$ is a positive transverse K\"ahler form. If:
\begin{equation}\label{hypothesis}
\frac {\bar s^T}2[d\eta]_B+(n-1)[\rho^T]_B>0,
\end{equation}
then the Mabuchi $K$-energy is bounded below on $[d\eta]_B$.
\end{theorem}
\begin{proof}
%Normalise $\eta$ in order to get $\bar s^T=1$. By hypothesis there exists $f\in C^{\infty}_B(M,\mathds{R})$ such that $\frac12d\eta_f-\rho^T>0$. 
Define a $J_{\chi}$ functional with $\chi=-\rho^T$. By Prop. \ref{jbound} and Theorem \ref{main}, condition \eqref{hypothesis} implies that this functional is bounded from below. Conclusion follows by observing that the Sasakian version of the Mabuchi $K$-energy map can be written as (see \cite[Prop. 3.2]{Z} ):
$$
\mathcal M(f)=\frac{\bar s^T}{n+1}I(f)+2J_{-\rho^T}(f)+2\int_M\ln\left(\frac{\eta\wedge (d\eta_f)^n}{\eta\wedge (d\eta)^n}\right)\eta\wedge (d\eta_f)^n,
$$
where from $f\in \mathcal H_0$ follows $I(f)=0$, and since $x\ln x>-e^{-1}$ for any $x>0$, the term $\int_M\ln\left(\frac{\eta\wedge (d\eta_f)^n}{\eta\wedge (d\eta)^n}\right)\eta\wedge (d\eta_f)^n$ is bounded below by $-e^{-1}\int_M\eta\wedge (d\eta)^n$.
\end{proof}


\begin{thebibliography}{48}
\bibitem{BHV} \textsc{L. Bedulli, W. He, L. Vezzoni}, Second order geometric flows on foliated manifolds, preprint 2015, arXiv:1505.03258v1 [math.DG].
\bibitem{bgbook} \textsc{C.P. Boyer, K. Galicki}, {\em Sasaki geometry}, Oxford Mathematical Monographs. Oxford University Press, Oxford, 2008. xii+613 pp.

\bibitem{cpz} \textsc{S. Calamai, D. Petrecca, K. Zheng}, The geodesic problem for the Dirichlet metric and the Ebin metric on the space of Sasakian metrics, preprint 2015, arXiv:1405.1211 [math.DG].

\bibitem{ca} \textsc{H-D. Cao}, Deformation of K\"ahler metrics to K\"ahler-Einstein metrics on compact K\"ahler manifolds, {\em Invent. Math.} {\bf81} (1985), 359--372.

\bibitem{chenM} \textsc{X. X. Chen}, On lower bound of the Mabuchi energy and its application. {\em Int. Math. Research Notices} {\bf 12} (2000).

\bibitem{chen} \textsc{X. X. Chen}, A new parabolic flow in K\"ahler manifolds, {\em Comm. An. Geom. } {\bf 12} (2004), n. 4, 837--852.

\bibitem{donaldsonNC} \textsc{S. K. Donaldson}, {\em Symmetric spaces, K\"ahler geometry and Hamiltonian dynamics}, in Northern California
Symplectic Geometry Seminar, Amer. Math. Soc. Transl. Ser. 2 196, Amer. Math. Soc., Providence, 1999, 13--33.

\bibitem{donaldsonMM} \textsc{S. K. Donaldson}, Moment maps and diffeomorphisms, \emph{Asian J. Math.} {\bf 3} (1999), no. 1, 1--15. 

\bibitem{F} \textsc{A. Futaki, H. Ono, G. Wang}, Transverse K\"ahler geometry of Sasaki manifolds and toric Sasaki-Einstein manifolds, {\em J. Diff. Geom.} {\bf83} (2009), 585--636.

\bibitem{guanzhang} \textsc{P. Guan, X. Zhang}, {\em A geodesic equation in the space of Sasakian metrics}, Geometry and analysis 1, 303--318,  Adv. Lect. Math. 17, Int. Press, Somerville, MA, 2011. 

\bibitem{guanzhangA} \textsc{P. Guan, X. Zhang}, Regularity of the geodesic equation in the space of Sasakian metrics, {\em Advances in Math.} {\bf230}, Issue 1 (2012), pp. 321--371.

\bibitem{hebey} E. Hebey, {\em Nonlinear Analysis on Manifolds: Sobolev Spaces and Inequalities}, AMS, Courant Ins. of Math. Sci. (2000).

\bibitem{he} \textsc{W. He}, On the transverse scalar curvature of a compact Sasaki manifold, {\tt  arXiv:1105.4000}, to appear in {\em Complex Manifolds}.  

\bibitem{Z} \textsc{X. Jin, X. Zhang}, Uniqueness of Constant Scalar Curvature Sasakian Metrics, preprint 2015, arXiv:1509.06522v2 [math.DG].

\bibitem{LSS}
\textsc{M. Lejmi, G, Sz\'ekelyhidi}, The J-flow and stability, {\em Adv. Math.} {\bf 274}, (2015), 404--431.

\bibitem{mabuchi} \textsc{T. Mabuchi}, Some symplectic geometry on compact K\"ahler manifolds. I,  {\em Osaka J. Math.} {\bf24} (1987), no. 2, 227--252. 

\bibitem{nitta} \textsc{Y. Nitta, K. Sekiya}, Uniqueness of Sasaki--Einstein metrics, {\em Tohoku Math. J.} (2)
{\bf 64}, n. 3 (2012), 453--468.

\bibitem{smoczyk} \textsc{K. Smoczyk, G. Wang, Y. Zhang}, The Sasaki-Ricci flow, {\em Internat. J. Math.} {\bf21} (2010), no. 7, 951--969. 


\bibitem{songweinkoveJ} \textsc{J. Song, B. Weinkove}, On the convergence and singularities of the $J$-flow with applications to the Mabuchi energy, preprint 2014.

\bibitem{sparks} \textsc{J. Sparks}, Sasakian-Einstein manifolds, {\em Surveys Diff.Geom.} {\bf 16} (2011), 265--324.  
%\bibitem{st} Streets, J., Tian, G.: A parabolic flow of pluriclosed metrics {\em Int. Math. Res. Not.} IMRN 2010, no. 16, 3101--3133. 
\bibitem{tian} \textsc{G. Tian}, 
On K\"ahler-Einstein metrics on certain K\"ahler manifolds with $C_1(M)>0$. 
{\em Invent. Math.} {\bf89} (1987), no. 2, 225--246.

\bibitem{tosatti} \textsc{V. Tosatti, B. Weinkove}, Estimates for the complex Monge-Amp\`ere equation on Hermitian and balanced manifolds, {\em Asian J. Math.} {\bf14} (2010), no. 1, 19--40. 

\bibitem{VZ} \textsc{L.\ Vezzoni, M.\ Zedda}, On the J-flow in Sasakian manifolds, to appear in {\em Ann. di Mat. Pura e Appl.}, 2015.

\bibitem{wein1} \textsc{B. Weinkove}, Convergence of the J-flow on K\"ahler surfaces, {\em Comm. Anal. Geom.} {\bf 12}, no. 4 (2004), 949-965.

\bibitem{wein2} \textsc{B. Weinkove}, On the J-flow in higher dimensions and the lower boundedness of the Mabuchi energy, {\em J. Differential Geom.} {\bf 73}, no. 2 (2006), 351--358.

\end{thebibliography}
\end{document}